\RequirePackage{fix-cm}
\documentclass[smallextended]{svjour3}       
\smartqed  
\usepackage{graphicx}
\usepackage{amsmath,amssymb,mathtools,bm}
\usepackage{tikz}
\usetikzlibrary{matrix,decorations.pathreplacing}
\usepackage{marginnote,hyperref}
\usepackage[misc,geometry]{ifsym} 
\newcommand{\norm}[1][\cdot]{\left\|#1\right\|}
\newcommand{\ip}[1][\cdot,\cdot]{\langle #1\rangle}

\newcommand{\ff}{\mathfrak{f}}
\newcommand{\fg}{\mathfrak{g}}
\newcommand{\bbC}{\mathbb{C}}
\newcommand{\C}{\mathbb{C}}
\newcommand{\bbR}{\mathbb{R}}
\DeclareMathOperator{\cay}{Cay}

\usepackage{xcolor}

\journalname{Journal of Fourier Analysis and Applications}
\begin{document}

\title{Gabor-type frames for signal processing on graphs}

\author{Mahya Ghandehari \and
        Dominique Guillot \and
        Kris Hollingsworth
}


\institute{M.~Ghandehari  \Letter \at
              Department of Mathematical Sciences \\
              University of Delaware\\
              \email{mahya@udel.edu}           
           \and
           D.~Guillot \at
              Department of Mathematical Sciences \\
		University of Delaware \\
		\email{dguillot@udel.edu}
	   \and 
		K. Hollingsworth \at
		School of Mathematics \\
		University of Minnesota \\
		\email{kgh@umn.edu}
}

\date{}

\maketitle

\begin{abstract}
In the past decade, significant progress has been made to generalize classical tools from Fourier analysis to analyze and process signals defined on networks. In this paper, we propose a new framework for constructing Gabor-type frames for signals on graphs. Our approach uses general and flexible families of linear operators acting as translations. Compared to previous work in the literature, our methods yield the sharp bounds for the associated frames, in a broad setting that generalizes several existing constructions. We also examine how Gabor-type frames behave for signals defined on Cayley graphs by exploiting the representation theory of the underlying group. We explore how natural classes of translations can be constructed for Cayley graphs, and how the choice of an eigenbasis can significantly impact the properties of the resulting translation operators and frames on the graph.

\end{abstract}

\section{Introduction}
\label{intro}
The recent field of graph signal processing was initiated to develop methods for analyzing signals defined on graphs.  
Given a graph $\Gamma$ on $N$ vertices, a graph signal  is a complex-valued function on the vertex set of $\Gamma$, which is naturally identified with a column vector in $\bbC^N$. 
A natural technique to analyze signals defined on graphs that is rapidly gaining popularity involves fixing a basis of eigenvectors for a chosen matrix associated with the graph, and expanding a given  graph signal in that basis.
The reason for doing so is to improve signal processing efficiency by working with a basis that is more adapted to the graph compared to an arbitrary basis of $\mathbb{C}^N$.
Natural examples of matrices associated with a graph $\Gamma$ include: (1) the adjacency matrix $A_{\Gamma}$ with entries $(i,j)$ equal to 1 when there is an edge from vertex $i$ to vertex $j$, and 0 otherwise; and (2) the graph Laplacian $L_{\Gamma}:=D_{\Gamma}-A_{\Gamma}$, where $D_{\Gamma}$ is the diagonal matrix with entry $d_{i,i}$ equal to the degree of vertex $i$.
Other matrices such as the normalized Laplacian and the random walk Laplacian have also been considered.
Such choices of orthonormal bases lead to the idea of graph Fourier analysis; see for example \cite{Hammond:2011:WaveletsOnGraphViaSpectralGraphTheory}.
%

A challenging task in graph signal processing is to produce efficient frames for the space of signals on a given graph. 
A \emph{frame} for an inner product space is a generalization of the notion of  basis, which provides a stable, possibly redundant system for analyzing vectors in that space. 
An important class of frames is constructed by applying a time-frequency shift operator to a given window function. Inspired by the seminal work of Gabor on $L^2(\mathbb{R})$ in  \cite{Gabor:1946:ThoeryOfCommunication}, such frames are called \emph{Gabor frames}. 
Wavelet frames constitute another notable class of frames that are closely related to Gabor frames.  Namely, Gabor frames are constructed through applications of translation and modulation operators to a window function, whereas the modulation operator is swapped with the dilation operator in the construction of wavelet frames. 

In this paper, we investigate Gabor-type constructions of frames for graph signals. 
Frame and wavelet constructions for graph signals have attracted the attention of many researchers in the past couple of decades. 
Early methods to construct frames based on the eigen-decomposition of the graph Laplacian are given by
Coifman and Maggioni in \cite{Coifman:2006:DiffusionWavelets} and by
Maggioni and Mhaskar in \cite{Maggioni:2008:DiffusionPolynomialFramesOnMetricMeasureSpaces}.
Frames associated with a Shannon-type sampling on graphs were initiated in \cite{pesenson1,pesenson2}.
Efforts to directly generalize multiresolution and wavelet analysis to the graph setting can be found in
\cite{Chui:2015:RepresentationOfFunctionsOnBigData:GraphsAndTrees,%
Crovella:2003:GraphWaveletsForSpecialTrafficAnalysis,%
Gavish:2010:MultiscaleWaveletsOnTrees,%
MR3932016,%
Jansen:2009:MultiscaleMethodsForDataOnGraphs,%
Lee:2008:Treelets,%
Murtagh:2007:HaarWaveletTransformOfADendogram}.
In \cite{Hammond:2011:WaveletsOnGraphViaSpectralGraphTheory},
Hammond, Vandergheynst and Gribonval define the graph Fourier transform and apply it to construct wavelet frames for graphs.
Other examples of wavelet-type frames based on the graph Fourier transform can be found in
\cite{Dong:2017:SparseRepresentationOnGraphsByTightWaveletFramesAndApplications,%
Leonardi:2013:TightWaveletFramesOnMultisliceGraphs,%
8683852,%
Shuman:2015:SpectrumAdaptedTightGraphWavelet,%
7873023}.
Studies exploring fundamental limits of how efficiently signals can be represented in terms of uncertainty principles can be found in
\cite{perraudin_ricaud_shuman_vandergheynst_2018,%
MR3538385},
and a proposed fast algorithm to implement frames on graphs can be found in \cite{JESTROVIC2017483}.
Some of the extensive work of defining Gabor-type frames in the graph setting, often referred to as vertex-frequency analysis, can be found in
\cite{MR3554601,%
MR3888082,%
Shuman:2012:WGFT,%
Shuman:2016:VertexFrequencyAnalysisOnGraphs,%
Shuman:2015:SpectrumAdaptedTightGraphWavelet,%
7472590,%
8108055,%
6842705}.
Summaries of most of the references mentioned (and many more) can be found in the survey articles
\cite{1705.02307,%
Ortega:2018:GSPOverview,%
1907.03471} or collected in the recent book \cite{MR3889009}.


In this article, we propose a general framework for constructing Gabor-type frames for signals on graphs. 
A major difficulty that arises in the construction of Gabor frames in the graph setting is the lack of a canonical notion of translation. Indeed, many notions of translations and shifts for signals on graphs have been defined in the literature so far, including: 
\begin{enumerate}
  \item the  translation operator  introduced by Shuman, Ricaud, and Vandergheynst \cite{Shuman:2016:VertexFrequencyAnalysisOnGraphs}. 
Inspired by classical (commutative) Fourier analysis, they define the notions of convolution, modulation, and translation via the graph Fourier transform;
  
  \item the linear isometric shift operator introduced by Girault, Gon\c{c}alves, and Fleury \cite{Girault:2015:TranslationsOnGraphs};

  \item the energy-preserving shift operator introduced by Gavili and Zhang \cite{Gavili:2017:OnTheShiftOperator};
  
  \item translation induced by the adjacency matrix of the graph, as proposed by Sandryhaila and Moura
  \cite{Sandryhaila:2013:DiscreteSignalProcessingOnGraphs};

  \item translation induced by pointwise multiplication with personalized PageRank vectors defined by Tepper and Sapiro \cite{7472590}, and;
  
  \item the neighborhood preserving translation defined by Pasdeloup
  et al.~\cite{Grelier:2016:NbhdPreservingTranslationsOnGraphs,Pasdeloup:2017:NbhdPreservingTranslationOperator}.
\end{enumerate}

A common feature of the above transformations is that they operate linearly on a given signal $\fg$. 
In this paper, we construct Gabor-type frames using general and flexible families of linear operators acting as translations. 
This viewpoint allows us to bring many previously defined natural graph frames under the same umbrella. 
That is, our frames (proposed in Theorem~\ref{TframeBound}) generalize many known frame constructions, for which we also obtain sharp frame bounds.  
See also \cite{erb} for generalizations of translation and modulation operators in a similar spirit.

The rest of the paper is organized as follows.
In Section \ref{background}, we collect the necessary background material on discrete frames, and provide a brief overview of signal processing on graphs. 
In Section \ref{section:main}, we present a general method for constructing  Gabor-type frames (Theorem~\ref{TframeBound}). We then provide the associated sharp frame bounds, and propose techniques for approximating frame bounds using certain matrix theoretic concepts. 
Next, we  devote our attention to the study of frames in which translations are defined using Fourier multipliers; 
examples include translations defined in  \cite{Gavili:2017:OnTheShiftOperator,Girault:2015:TranslationsOnGraphs,Sandryhaila:2013:DiscreteSignalProcessingOnGraphs,Shuman:2016:VertexFrequencyAnalysisOnGraphs}.  We show that our general approach leads to sharp frame bounds for this class of frames (Theorem~\ref{TframeOrtho}). This provides a unified proof for  some formerly known frames in the literature such as \cite{Gavili:2017:OnTheShiftOperator,Shuman:2016:VertexFrequencyAnalysisOnGraphs}, and allows us to compute sharp frame bounds in each case.
In Section \ref{EigenbasisForQuasiAbelian}, we examine the constructed frames in the special case where the graph $\Gamma$ is a {\it Cayley graph}.  In that case, an orthonormal basis of eigenvectors of the adjacency matrix or the  Laplacian of the graph can be explicitly obtained by exploiting the representation theory of the associated group \cite{Babai:1979:SpectraOfCayleyGraphs}. Building on the work in \cite{MDK:2019:Sampta}, we  study how properties of the frames given above relate to the structure of the underlying group. 
For example, in Theorem \ref{TCayleyDiag} we show that the condition (on the window function) of Corollary \ref{CDiagTight}  for producing a tight frame can be considerably relaxed in the case of a Cayley graph.
Finally, we use Cayley graphs to demonstrate the importance of carefully choosing a basis of eigenvectors associated to the graph in the case where repeated eigenvalues occur.

\section{Notations and Background}
\label{background}
\noindent{\bf Discrete Frames.}
A \emph{discrete frame} for a separable Hilbert space $\mathcal H$ is a set of vectors $\{\phi_x\}_{x\in X}$ indexed by a countable set $X$, such that for some positive real numbers $A$ and $B$, we have 
\begin{equation}\label{Eframedefn}
  A\|f\|_{\mathcal H}^2
    \leq \sum_{x\in X} |\langle f, \phi_x \rangle |^2
    \leq B\|f\|_{\mathcal H}^2, \mbox{ for every vector }f\in \mathcal H.%
\end{equation}%
Frames provide stable, possibly redundant systems which allow reconstruction of a signal $f$ from its frame coefficients $\{\langle f, \phi_x \rangle\}_{x\in X}$. 
When the frame provides redundant representation, reconstruction of a signal is still possible even when some portion of its frame coefficients is lost or corrupted.

We define the frame \emph{condition number} of a frame $\mathcal{F}$ as the ratio $c(\mathcal{F}) := B/A$, where $A, B$ denote the optimal constants satisfying Equation \eqref{Eframedefn}. 
An important class of frames is the class of \emph{tight frames}, {\it i.e.}, frames for which $A=B$.
These frames exhibit many desirable properties, such as greater numerical stability when reconstructing noisy signals, compared to general frames or to orthonormal bases. 
For example, one can show that, under natural assumptions, the mean-square error of a reconstruction is minimized if and only if the frame is tight (see \cite{Christensen:2016:IntroductionToFrames}, Theorem 1.9.2).
An important goal in designing frames for real applications is to design tight frames, or at least frames with a small condition number.
For a detailed introduction to frame theory, see \cite{Christensen:2016:IntroductionToFrames,HanLarson:FramesIntro:2000}.\\

\noindent{\bf Gabor frames.} A natural approach to construct frames involves applying a time-frequency shift operator to a given function $g$. In his seminal 1946 paper, Gabor \cite{Gabor:1946:ThoeryOfCommunication} proposes constructing such frames for functions in $L^2(\mathbb{R})$ by defining
\[
g_{u, \xi}(t) := (M_\xi T_u g)(t) = e^{2\pi i \xi t}g(t-u), 
\]
where $(T_u g)(t) = g(t-u)$ and $(M_\xi g)(t) = e^{2\pi i \xi t} g(t)$ denote the standard translation and modulation operators on $L^2(\mathbb{R})$. Such frames are commonly used in science and engineering and have been extensively studied -- see \cite{casazza2012finite} for more details. \\


\noindent{\bf Graph Signal Processing.}
Let $\Gamma$ be a graph with vertex set $V(\Gamma) = \{1, 2, \ldots, N\}$.
A signal on $\Gamma$ is a function $\mathfrak f:V(\Gamma)\to \mathbb{C}$.
We identify the signal $\mathfrak f$ with the vector $(\mathfrak f(1), \mathfrak f(2), \ldots, \mathfrak f(N))^\top$ in $\C^N$, where $M^\top$ denotes the transpose of the matrix $M$.

To develop signal processing for a  given undirected graph $\Gamma$ with $N$ vertices, we first fix an associated graph matrix. 
The most significant matrices associated with a graph $\Gamma$ are the adjacency matrix $A_\Gamma$ and the Laplacian matrix $L_\Gamma$.
Let $\{\phi_j\}_{j=1}^{N}$ be an orthonormal basis of eigenvectors for the chosen matrix, associated to (repeated) eigenvalues $\{\lambda_j\}_{j=1}^{N}$.
Inspired by commutative Fourier analysis,
the \emph{graph Fourier transform} was introduced by Hammond, Vandergheynst, and Gribonval in \cite{Hammond:2011:WaveletsOnGraphViaSpectralGraphTheory} as the expansion of the vector $\mathfrak f\in\mathbb C^N$ in terms of the orthonormal basis $\{\phi_j\}_{j=1}^{N}$.
More precisely, the Fourier coefficients of $\mathfrak{f}$ are given by
\begin{equation}\label{GFT}
\widehat{\mathfrak f}(\phi_k) = \ip[\mathfrak f, \phi_k]_{\mathbb{C}^N} = \sum_{j=1}^N \mathfrak f(j)\overline{\phi_k(j)}.
\end{equation}%
Equivalently, letting $\Phi$ be the matrix whose $j^{\textnormal{th}}$ column is $\phi_j$, we have $\widehat{\mathfrak{f}}=\Phi^*\mathfrak f$.
The \emph{inverse graph Fourier transform} is then given by $\mathfrak f=\Phi \widehat{\mathfrak{f}}$, or%
\begin{equation}\label{GIFT}
\mathfrak f(k) = \sum_{j=1}^{N} \widehat{\mathfrak f}(\phi_j) \phi_j(k).
\end{equation}
Note that here, we use the notation $\widehat{\mathfrak f}(\phi_k)$, rather than the more conventional notation $\widehat{\mathfrak f}(\lambda_k)$ to avoid confusion in cases where repeated eigenvalues occur. 
See
\cite{%
Ortega:2018:GSPOverview,%
Sandryhaila:2013:DiscreteSignalProcessingOnGraphs,%
Sandryhaila:2014:BigDataAnalysisWithSPoG,%
Shuman:2013:EmergingFieldOfSignalProcessing}
for more details on the graph Fourier transform and the associated theory.

The graph Fourier transform can be used to generalize the concepts of convolution, modulation, and translation to the graph setting. To elaborate, define the {\it convolution} of two signals $\ff, \fg$ on $\Gamma$ to be the pointwise product in the Fourier domain
\begin{equation}\label{Econv}
\mathfrak f*\mathfrak g=\Phi(\widehat{\mathfrak{f}} \circ \widehat{\mathfrak{g}}), 
\end{equation}
where we use $\circ$ to denote entry-wise (Hadamard) multiplication of matrices.  This convolution naturally leads to a notion of translation by defining 
\begin{equation}\label{Etranslation}
T_j \mathfrak f= \sqrt{N}(\mathfrak{f}*\delta_j) \qquad (j=1,\dots,N), 
\end{equation}
where $\delta_j$ denotes the Kronecker delta function centered at vertex $j$, {\it i.e.}, 
\[
\delta_j(k) = \begin{cases}
1 & \textrm{if } k=j \\
0 & \textrm{otherwise}. 
\end{cases}
\]
Here the factor of $\sqrt{N}$ is used so that the graph translation preserves the mean of the signal when in the setting of \cite{Shuman:2016:VertexFrequencyAnalysisOnGraphs}.
Finally, signal modulation is defined as entrywise multiplication with the basis functions:
\begin{equation}\label{Emodulation}
M_j\mathfrak f = \phi_j\circ \mathfrak f \qquad (j=1,\ldots,N).
\end{equation}
Using these definitions, Shuman et al.~\cite{Shuman:2016:VertexFrequencyAnalysisOnGraphs} defined a frame for graph signals that is analogous to the classical construction of Gabor frames on the real line. 
Given a window function $\mathfrak g:V( \Gamma)\to \mathbb{C}$, let
\begin{equation}\label{Efrakg}
\mathfrak{g}_{j,k}:=M_kT_j\mathfrak g \qquad (j,k=1,\ldots,N).
\end{equation}
One of the main results of \cite{Shuman:2016:VertexFrequencyAnalysisOnGraphs} is the fact that, under mild assumptions, the functions $\mathfrak{g}_{j,k}$ define a frame that can be used to analyze real signals on $\Gamma$. 
\begin{theorem}[\cite{Shuman:2016:VertexFrequencyAnalysisOnGraphs}, Theorem 3]
\label{thm:OriginalShumanTheorem}
Let $\Gamma$ be a graph. 
Let $\{\phi_j\}_{j=1}^N$ be an orthonormal basis of real eigenvectors for the graph Laplacian matrix, and let $\fg \in \bbR^N$. If $\sum_{j=1}^{N}\mathfrak g(j)\neq 0$ then the collection of functions $\{\mathfrak{g}_{j,k}\}_{j,k=1,\ldots,N}$ is a frame for $\mathbb{R}^N$, {\it i.e.}, for all $\mathfrak{f}\in \mathbb{R}^N$,%
\begin{equation}\label{eqn:OriginalShumanBounds}
    \min_{n=1}^N \{N\norm[T_n\mathfrak g]_2^2\} \norm[\mathfrak f]_2^2
    \leq
    \sum_{i=1}^N
    \sum_{k=0}^{N-1}
    |\ip[\mathfrak f, \mathfrak g_{i,k}]|^2
    \leq
    \max_{n=1}^N \{N\norm[T_n\mathfrak g]_2^2\} \norm[\mathfrak f]_2^2
\end{equation}
\end{theorem}
The proof of the above theorem relies on calculations which hold true only in the space of real-valued vectors. Moreover the statement is specific to the particular definitions of translations and modulations  in \cite{Shuman:2016:VertexFrequencyAnalysisOnGraphs}. In the next section, we obtain a generalization of this theorem where the translations $T_j$ are replaced by arbitrary linear operators. Our results hold for both real and complex-valued signals, and are independent of the matrix (adjacency matrix, Laplacian, etc.) associated to the graph. We also derive the sharp frame bounds of the generalized frames.

\section{General Constructions of Gabor-Type Frames}
\label{section:main}
In this section, we present a general method for constructing  Gabor-type frames, and provide the associated sharp frame bounds. We also demonstrate how the frame bounds and the frame condition number can be estimated via generalized eigenvalue problems. We finish this section with the study of frames in which the translation operator is defined via Fourier multipliers. We show that our general approach leads to sharp frame bounds for this class of frames; this provides a unified proof for  several formerly known frames in the literature, and allows us to compute frame bounds for each case. Finally, we discuss the case where the translation operators are given as Fourier multipliers of an orthonormal set of vectors.

\begin{theorem}\label{TframeBound}
Let $\{\phi_j\}_{j=1}^N$ be an orthonormal basis of $\mathbb{C}^N$, let $A_1, A_2, \dots, A_S $ be an arbitrary collection of complex $N \times N$ matrices, and let $\fg \in \mathbb{C}^N$.
For $m=1,\dots,S$ and $\ell=1,\dots N$, define
\begin{equation}\label{Eframe}
  \fg_{m,\ell} := \phi_\ell \circ (A_m \mathfrak{g}), 
\end{equation}
where $\circ$ denotes the entrywise product.
Also let
\begin{equation}\label{Ev}
  v = (v_k)_{k=1}^N:= \sum_{j=1}^S |A_j \mathfrak{g}|^2, 
\end{equation}
where the modulus and square operations are performed entrywise.
Then the collection of vectors $\{\mathfrak{g}_{m,\ell}: m=1,\dots,S, \ell=1,\dots,N\}$ forms a frame for ${\mathbb C}^N$  if and only if $v_k> 0$ for all $k=1,\dots, N$.
Moreover, in that case, 
\[
  A \, \|\mathfrak{f}\|_2^2 \leq \sum_{m=1}^S \sum_{\ell=1}^N |\langle \mathfrak{f}, \fg_{m,\ell}\rangle|^2 \leq B \, \|\mathfrak{f}\|_2^2
\]
with optimal frame bounds $A := \min_{k=1}^N v_k$ and $B:= \max_{k=1}^N v_k$.
\end{theorem}
\begin{proof}
For $i=1, \dots, N$, let $D_i$ denote the diagonal matrix with $k$-th diagonal entry equal to the $k$-th term of the vector $\phi_i$.
Using that notation, observe that $\fg_{m,\ell} = D_\ell A_m \mathfrak{g}$.
Now, consider the matrix whose columns are the vectors $\fg_{m,\ell}$: 
\[
  T :=
  \begin{pmatrix}
    D_1 A_1 \fg, D_1 A_2 \fg, \dots, D_1 A_S \fg, D_2 A_1 \fg, D_2 A_2 \fg, \dots, D_2 A_S \fg, \dots, D_N A_S \fg
  \end{pmatrix}.
\]
By standard results on finite frame theory, the collection of vectors $\{\mathfrak{g}_{m,\ell}: m=1,\dots,S, \ell=1,\dots,N\}$ is a frame if and only if the matrix $TT^*$ is positive definite.
Moreover, the associated optimal frame bounds are given by the smallest and largest eigenvalues of $TT^*$ (see {\it e.g.}~\cite[Theorem 1.3.1]{Christensen:2016:IntroductionToFrames}).
Here, we have 
\[
  TT^*
    =
    \sum_{i=1}^N \sum_{j=1}^S D_i A_j \fg \fg^* A_j^* D_i^*.
\]
Now, observe that for any diagonal matrix $D = \textrm{diag}(u)$ and any matrix $M$, we have $DMD^* = M \circ (uu^*)$. Hence, 
\begin{align*}
TT^* &= \sum_{i=1}^N \sum_{j=1}^S \left[(A_j \fg)(A_j \fg)^*\right] \circ (\phi_i \phi_i^*) \\
&=\sum_{j=1}^S \left[(A_j \fg)(A_j \fg)^*\right] \circ \left(\sum_{i=1}^N \phi_i \phi_i^*\right) \\
& =\sum_{j=1}^S \left[(A_j \fg)(A_j \fg)^*\right] \circ I_N, 
\end{align*}
where $I_N$ denotes the $N \times N$ identity matrix and where the last line follows from the fact that the $\phi_i$'s form an orthonormal basis of $\mathbb{C}^N$. Hence $TT^*$ is diagonal with diagonal entries given by $\sum_{j=1}^S |A_j \fg|^2$. The result now follows immediately from 
\cite[Theorem 1.3.1]{Christensen:2016:IntroductionToFrames}
\end{proof}

\begin{remark}\hfill
\begin{itemize}
\item[1.] Our construction in Theorem \ref{TframeBound} holds for any orthonormal basis  $\{\phi_j\}_{j=1}^N$ of ${\mathbb C}^N$. Thus, the theorem is valid regardless of the particular graph matrix one may choose to analyze graph signals. 
\item[2.] The above theorem generalizes Theorem \ref{thm:OriginalShumanTheorem} to allow a general set of linear operators as translations. In addition, it yields a frame for ${\mathbb C}^N$ rather than just ${\mathbb R}^N$. 
\item[3.] We note that the vector $v$ in Theorem~\ref{TframeBound} captures information about how the translated windows are spread across the graph. In particular, in order to obtain a frame, at least one of the translated windows should overlap with each vertex.
\end{itemize}
\end{remark}

\begin{corollary}
In the same setting as Theorem \ref{TframeBound}, the family 
\[
  \{\mathfrak{g}_{m,\ell}: m=1,\dots,S, \ell=1,\dots,N\}
\] 
forms a frame if and only if  for every $1 \leq k \leq N$ there exists $1 \leq j \leq S$ such that $(A_j \fg)_k \ne 0$.
\end{corollary}

While Theorem \ref{TframeBound} provides explicit frame bounds for \eqref{Eframe}, it is not immediately clear how the entries of the vector $v$ in Equation \eqref{Ev} vary with the vector $\fg$. The following result provides a different description of the entries of $v$ that clarifies this relationship. 

\begin{theorem}\label{Ttight}
Consider the same setting as Theorem \ref{TframeBound} with $A_j := (a^{(j)}_{k\ell})_{k,\ell=1}^N$ and $v \in \C^N$ as in Equation \eqref{Ev}.
For $k,\ell=1,\dots, N$, define $\mathbf w_{k,\ell} \in \mathbb{C}^S$ by
\[
  \mathbf w_{k,\ell} := (a^{(j)}_{k\ell})_{j=1}^S, 
\]
and let 
\begin{equation}\label{ECk}
  C_k 
    := (\langle \mathbf w_{k,\ell}, \mathbf w_{k,m}\rangle)_{\ell,m=1}^N
    = \left(\sum_{j=1}^S a_{k\ell}^{(j)}  \overline{a_{km}^{(j)}}\right)_{\ell,m=1}^N \in \mathbb{C}^{N \times N}.
\end{equation}
Then $v_k = \fg^* C_k \fg$ for any $1 \leq k \leq N$.
In particular, the family of vectors $\fg_{m,\ell}$ forms a frame if and only if $\fg \not\in \cup_{k=1}^N \ker C_k$. 
\end{theorem}
\begin{proof}
The $k^{\textnormal{th}}$ entry of $v$ is given by 
\begin{align*}
  v_k
    &= \sum_{j=1}^S \left|\sum_{\ell=1}^N a_{k\ell}^{(j)} \fg_\ell \right|^2 = \sum_{j=1}^S \left(\sum_{\ell=1}^N a_{k\ell}^{(j)} \fg_\ell\right) \overline{\left(\sum_{m=1}^N a_{km}^{(j)} \fg_m\right)} \\
    &= \sum_{j=1}^S \sum_{\ell,m=1}^N a_{k\ell}^{(j)}  \overline{a_{km}^{(j)}} \fg_\ell \overline{\fg_m} \\
    &= \fg^* C_k \fg, 
\end{align*}
where $C_k$ is as in Equation \eqref{ECk}. 
\end{proof}

\noindent
Note that, for each $k$, the matrix $C_k$ in Theorem \ref{Ttight} is the Gram matrix generated by the vectors $\{\mathbf{w}_{k,j}\}_{j=1}^N$, therefore each is a positive semidefinite Hermitian matrix.
Using Theorem \ref{Ttight}, we immediately obtain useful estimates on the frame bounds given in Theorem \ref{TframeBound}, as well as the resulting condition number of the frame.
Given a symmetric matrix $M$, denote by $\lambda_{\min}(M)$ and $\lambda_{\max}(M)$ the smallest and largest eigenvalues of $M$, respectively.
The following result provides estimates on the frame bounds that are independent of $\fg$. 
In particular, in the case where the translations are given as in Theorem \ref{thm:OriginalShumanTheorem}, we obtain the best frame bounds possible that are valid for any given unit vector $\fg$.

\begin{corollary}\label{Ccond}
Consider the same setting as Theorem \ref{TframeBound}, and assume furthermore that $\|\fg\| = 1$. Then 
\[
  \min_{k=1}^N \lambda_{\min}(C_k) \; \|\ff\|_2^2 
    \leq \sum_{m=1}^S \sum_{\ell=1}^N |\langle \ff, \fg_{m,\ell}\rangle|^2 
    \leq \max_{k=1}^N \lambda_{\max}(C_k) \; \|\ff\|_2^2, 
\]
In particular, if $\min_{k=1}^N \lambda_{\min}(C_k) > 0$, then $\mathcal{G} := \{\fg_{m,\ell}: m=1,\dots,S, \ell=1,\dots,N\}$ forms a frame whose condition number $c(\mathcal{G})$ satisfies
\[
  c(\mathcal{G}) \leq \frac{\max_{k=1}^N \lambda_{\max}(C_k)}{\min_{k=1}^N \lambda_{\min}(C_k)}.
\]
\end{corollary}

The analysis of the condition number of the frame $\mathcal{G}$ given in Corollary \ref{Ccond} can be further refined as follows.
Recall that $\lambda \in \C$ is said to be an {\it eigenvalue} of the matrix pencil $A-zB$ if 
\[
  \det(A-\lambda B) = 0.
\]
In that case, there exists $v \in \C^N \setminus \{0\}$ such that $Av = \lambda Bv$.
We will focus on the case where $A$ is Hermitian and $B$ is positive definite below.
In that case, the eigenvalue problem for the pencil is equivalent to the standard eigenvalue problem
\[
  B^{-1/2}A B^{-1/2} v = \lambda v. 
\]
As a consequence, the pencil has exactly $n$ real eigenvalues $\lambda_1 \leq \dots \leq \lambda_N$ that can be computed via the Courant-Fischer min-max principles: 
\begin{align*}
  \lambda_j
    & = \min_{\dim U = j} \max_{{0\neq}u \in U} \frac{u^*Au}{u^*Bu}
      = \max_{\dim U = N-j+1} \min_{{0\neq}u \in U}  \frac{u^*Au}{u^*Bu}. 
\end{align*}
In particular, 
\begin{equation}\label{Eminmax}
  \lambda_1 = \min_{{0\neq}u \in \C^N}  \frac{u^*Au}{u^*Bu}  \qquad \lambda_N = \max_{{0\neq}u \in \C^N}  \frac{u^*Au}{u^*Bu} .
\end{equation}
See {\it e.g.}~\cite{ikramov1993matrix,li2015rayleigh} for more details about matrix pencils. 

As a consequence of the above discussion, we immediately obtain the following sharp upper bound on the condition number of the frame $\{\fg_{m,\ell}\}$, under the assumption that $\|\fg\|=1$.
For two Hermitian positive definite matrices $A, B$, let $\lambda_{\min}(A,B)$ and $\lambda_{\max}(A,B)$ denote the smallest and largest eigenvalues of the pencil $A-zB$ respectively. 

\begin{theorem}
Let $\mathcal{G} := \{\mathfrak{g}_{m, \ell}: m=1,\ldots,S, \ell=1,\ldots,N\}$ with $\fg_{m,\ell}$ as in Theorem \ref{TframeBound}.
Let $C_k$ be as in Equation \eqref{ECk} and assume $C_1, \dots, C_N$ are positive definite.
Then
\begin{equation}
\label{Esupcond}
  \sup_{\|\fg\|=1} c(\mathcal{G}) =  \max_{k, \ell=1, \dots, N} \lambda_{\max}(C_k, C_\ell).
\end{equation}
Equality is attained when $\fg$ is an eigenvector associated to the generalized eigenvalue problem $C_{k^*} - \lambda C_{\ell^*}$, where $k^*$ and $\ell^*$ are values of $k$ and $\ell$ attaining the maximum in Equation \eqref{Esupcond}.
\end{theorem}
\begin{proof}
By Theorems \ref{TframeBound} and \ref{Ttight}, we have 
\[
  c(\mathcal{G})
    = \frac{\max_{k=1}^N \fg^* C_k \fg}{\min_{\ell=1}^N \fg^* C_\ell \fg}
    = \max_{k,\ell=1, \dots, N}  \frac{\fg^* C_k \fg}{\fg^* C_\ell \fg}.
\]
The result follows from Equation \eqref{Eminmax} upon maximizing over $\fg$. 
\end{proof}

\subsection{Frames defined via Fourier multipliers}

A common feature of several translation operators for graph signals  that have been proposed in the literature is that they operate by entry-wise multiplication in the Fourier domain, {\it i.e.}, 
\[
\widehat{T \fg} = \widehat{\fg} \circ \widehat{\ff} 
\]
for some $\ff \in \C^N$. Using the notion of convolution defined in Equation \eqref{Econv}, this is equivalent to $T \fg = \fg * \ff$. Equivalently, for a given vector $w \in \C^N$, denote by $D_w$ the diagonal matrix with diagonal entries $w_1, w_2, \dots, w_N$. Then the above operator can be written as
\[
  T  = \Phi D_{\widehat{f}} \Phi^*,
\]
and is called a \emph{Fourier multiplier}. 
In fact, the first four examples of translation/shift operators given in Section \ref{intro} are special instances of Fourier multipliers.
Other examples include translations or shift obtained by applying functions to the Laplacian of the graph via the functional calculus.

The following result provides explicit frame bounds when  translations are defined as Fourier multipliers.

\begin{theorem}\label{TframeOrtho}
Let $\{\phi_j\}_{j=1}^N$ be an orthonormal basis of $\mathbb{C}^N$, let $\{f_j\}_{j=1}^S$ be an arbitrary collection of vectors in $\mathbb{C}^N$, and let $\fg \in \mathbb{C}^N$.
Define
\[
  A_i = \Phi D_{f_i} \Phi^* \qquad (i=1,\dots,S).
\] 
For $m=1,\dots, S$ and $\ell=1,\dots,N$, define $\fg_{m,\ell}$ as in Equation \eqref{Eframe}.
Let $F$ be the $N \times S$ matrix whose $j$-th columns is $f_j$ and let $\mu_j := (\phi_k(j))_{k=1}^N \in \mathbb{C}^N$.
Then we have
\[
A \, \|\mathfrak{f}\|_2^2 \leq \sum_{m=1}^{S} \sum_{l=1}^N |\langle \mathfrak{f}, \fg_{m,l}\rangle|^2 \leq B \, \|\mathfrak{f}\|_2^2
\]
where 
\begin{align*}
A = \min_{k=1}^N \|F^*(\mu_k \circ \widehat{\overline{\fg}})\|^2 \qquad \textrm{and}\qquad B = \max_{k=1}^N \|F^*(\mu_k \circ \widehat{\overline{\fg}})\|^2.
\end{align*}
The constants $A$ and $B$ are sharp. 
\end{theorem}
\begin{proof}
We compute the vector $v$ in Theorem \ref{TframeBound}. We have 
\[
v = \sum_{i=1}^{S} |A_i \fg|^2 = \sum_{i=1}^{S} |\Phi D_{f_i} \Phi^* \fg|^2 = \sum_{i=1}^{S} |\Phi (f_i \circ \widehat{\fg})|^2.
\]
Let $F = (f_1, \dots, f_S)$ be the matrix whose columns are $f_1, \dots, f_S$. 
Then 
\begin{align*}
v_k
  =\sum_{i=1}^S \left|\ip[\mu_k, f_i \circ \widehat \fg]\right|^2
  & =\sum_{i=1}^S \left|\ip[f_i,  \mu_k \circ \overline{\widehat \fg}]\right|^2 \\
  & = \sum_{i=1}^S \ip[f_i, \mu_k \circ \overline{\widehat \fg}] \overline{\ip[f_i, \mu_k \circ \overline{\widehat \fg}]} \\
  &= (\mu_k \circ \overline{\widehat{\fg}})^* FF^*(\mu_k \circ \overline{\widehat{\fg}}) \\
  &= \|F^*(\mu_k \circ \overline{\widehat{\fg}})\|^2.
\end{align*}
The result now follows from Theorem \ref{TframeBound}. 
\end{proof}

\begin{corollary}\label{CframeOrtho}
Assume the functions $\{f_j\}_{j=1}^N$ are orthonormal in Theorem \ref{TframeOrtho}.
Then the family of functions $\{\fg_{m,\ell} : m,\ell=1,\dots,N\}$ forms a frame if and only if for every $1 \leq k \leq N$ there exists $1 \leq j \leq N$ such that $\phi_j(k)$ and $\widehat{\fg}(j)$ are both non-zero.
In that case, the sharp bounds for the associated frame are given by 
\begin{align*}
  A 
    = \min_{k=1}^N  \sum_{j=1}^N |\phi_j(k)|^2\cdot |\widehat{\fg}(j)|^2 \qquad\textrm{and}\qquad  B
    = \max_{k=1}^N  \sum_{j=1}^N |\phi_j(k)|^2\cdot |\widehat{\fg}(j)|^2.
\end{align*}
\end{corollary}
In particular, observe that the frame is tight when $\widehat{\fg}$ is constant.

\begin{corollary}\label{CDiagTight}
Assume the functions $\{f_j\}_{j=1}^N$ are orthonormal in Theorem \ref{TframeOrtho}.
Moreover, assume $\widehat{\fg}$ is constant.
Then the family of functions $\{\fg_{m,\ell} : m,\ell=1,\dots,N\}$ forms a tight frame. 
\end{corollary}
\begin{proof}
Assume $\widehat{\fg} \equiv c$ for some $c \in \C$. Then 
\[
  A 
    = |c|^2 \cdot \min_{k=1}^N  \sum_{j=1}^N |\phi_j(k)|^2
    = |c|^2
\]
since the $\{\phi_j\}_{j=1}^N$ are orthonormal. Similarly, we obtain $B = |c|^2$. 
\end{proof}

Interestingly,  the frame bounds in Corollary \ref{CframeOrtho} are independent of the choice of the vectors $\{f_i\}_{i=1}^N$, as long as they are orthonormal. Using a trivial estimate on the above optimal frame bound and the orthonormality of the rows of $\Phi$, we immediately obtain the following estimates that are independent of the basis $\{\phi_j\}_{j=1}^N$. 
\begin{corollary}
Under the assumptions of Theorem \ref{TframeOrtho}, we have 

\noindent
\resizebox{\textwidth}{!}{%
$\displaystyle
  A 
    = \min_{k=1}^N \|F^*(\mu_k \circ \widehat{\fg})\|^2
    \geq \lambda_{\min}(FF^*) \cdot\min_{k=1}^N  \sum_{j=1}^N |\phi_j(k)|^2\cdot |\widehat{\fg}(j)|^2
    \geq  \lambda_{\min}(FF^*) \cdot \min_{k=1}^N |\widehat{\fg}(k)|^2,
$
}
\noindent
\resizebox{\textwidth}{!}{%
$\displaystyle
  B 
    = \max_{k=1}^N \|F^*(\mu_k \circ \widehat{\fg})\|^2
    \leq \lambda_{\max}(FF^*)\cdot \max_{k=1}^N  \sum_{j=1}^N |\phi_j(k)|^2\cdot |\widehat{\fg}(j)|^2
    \leq \lambda_{\max}(FF^*) \cdot \max_{k=1}^N |\widehat{\fg}(k)|^2.
$
}
In particular, if $FF^*$ is non-singular and $\min_{k=1}^N  \sum_{j=1}^N |\phi_j(k)|^2\cdot |\widehat{\fg}(j)|^2 > 0$, then $\mathcal{G} := \{\fg_{m,\ell} : m=1,\dots,S,\;\ell=1,\dots,N\}$ forms a frame whose condition number satisfies: 
\[
{c}(\mathcal{G}) \leq \kappa(FF^*) \cdot \frac{\max_{k=1}^N  \sum_{j=1}^N |\phi_j(k)|^2\cdot |\widehat{\fg}(j)|^2}{\min_{k=1}^N  \sum_{j=1}^N |\phi_j(k)|^2\cdot |\widehat{\fg}(j)|^2} \leq \kappa(FF^*) \cdot \frac{\max_{k=1}^N |\widehat{\fg}(k)|^2}{\min_{k=1}^N |\widehat{\fg}(k)|^2}, 
\]\
where $\kappa(FF^*)$ denotes the condition number of the matrix $FF^*$.
\end{corollary}

Next, we illustrate how Theorem \ref{TframeOrtho} can be applied to yield the sharp frame bounds for the frames provided by Theorem \ref{thm:OriginalShumanTheorem} (introduced in 
\cite{Shuman:2016:VertexFrequencyAnalysisOnGraphs}).
It is worth noting that Theorem \ref{TframeOrtho} shows, in addition, that the construction proposed in Theorem \ref{thm:OriginalShumanTheorem} produces a frame for ${\mathbb C}^N$, rather than just ${\mathbb R}^N$.

\begin{corollary}[Sharp frame bounds for frames from \cite{Shuman:2016:VertexFrequencyAnalysisOnGraphs}]\label{CTransShuman}
Let $\{\phi_j\}_{j=1}^N$ be an arbitrary orthonormal basis of $\mathbb{C}^N$ and let $\fg \in \C^N$. For $i=1,\dots,N$, define
\begin{equation}\label{ETShuman}
T_i \fg := \fg * (\sqrt{N} \delta_i).
\end{equation}
For $m,\ell=1,\dots,N$, define $\fg_{m,\ell}$ as in Equation \eqref{Eframe}, where $T_i$ plays the role of $A_i$.
Then we have
\[
  A \, \|\mathfrak{f}\|_2^2
    \leq \sum_{m=1}^N \sum_{\ell=1}^N |\langle \mathfrak{f}, \fg_{m,\ell}\rangle|^2
    \leq B \, \|\mathfrak{f}\|_2^2
\]
where 
\begin{align*}
  A 
    = N \cdot \min_{k=1}^N  \sum_{j=1}^N |\phi_j(k)|^2\cdot |\widehat{\fg}(j)|^2
    \qquad\textrm{and}\qquad
  B 
    = N \cdot \max_{k=1}^N  \sum_{j=1}^N |\phi_j(k)|^2\cdot |\widehat{\fg}(j)|^2.
\end{align*}
The constants $A$ and $B$ are sharp. Moreover, these constants coincide with the frame bounds given in Theorem \ref{thm:OriginalShumanTheorem} (i.e. \cite{Shuman:2016:VertexFrequencyAnalysisOnGraphs}, Theorem 3).
\end{corollary}
\begin{proof}
Observe that $T_i \fg = \Phi D_{\sqrt{N}\widehat{\delta_i}} \Phi^*$. Now, for $k=1,\dots, N$, we have
\[
  \widehat{\delta_i}(k)
    = \langle \delta_i, \phi_k\rangle
    = \sum_{j=1}^N \delta_i(j) \overline{\phi_k(j)}
    = \overline{\phi_k(i)}.
\]
It follows easily that $\widehat{\delta_1}, \dots, \widehat{\delta_N}$ is a orthonormal basis of $\C^N$.
The result now follows immediately from Corollary \ref{CframeOrtho}. The fact that $A$ and $B$ are the same as the frame bounds in Theorem \ref{thm:OriginalShumanTheorem} follows from an easy application of Parseval's identity.
\end{proof}

Finally, we apply Theorem \ref{TframeOrtho} to obtain sharp bounds for the frame construction given by repeatedly applying the energy preserving shift operator of Gavili and Zhang \cite{Gavili:2017:OnTheShiftOperator} to a signal $\fg$. Let $A = \Phi\Lambda \Phi^*$ denote the eigen-decomposition of the adjacency matrix $A$ of the graph $\Gamma$. The authors in \cite{Gavili:2017:OnTheShiftOperator} define the shift operator $A_\alpha$ by 
\begin{equation}\label{EshiftEnergy}
A_\alpha = \Phi D_\alpha \Phi^*, 
\end{equation}
where $\alpha \in \C^N$ is an arbitrary vector of distinct complex numbers of modulus $1$. Of particular interest is the case where $\alpha_k \overline{\alpha_\ell} = e^{-i\frac{2\pi(k-\ell)}{N}}$, {\it i.e.}, 
\begin{equation}\label{Ephase}
\alpha_k= e^{i\left(c - \frac{2\pi(k-1)}{N}\right)}
\end{equation}
where $c \in [0,2\pi)$.
Note that \cite{Gavili:2017:OnTheShiftOperator} only considers $c=0$ and observe that, under this assumption, the shift operator $A_\alpha$ given by \eqref{EshiftEnergy} satisfies $A_\alpha^N = I$.
\begin{corollary}[Sharp frame bounds for frames from \cite{Gavili:2017:OnTheShiftOperator}]\label{CTransGavili}
Let $\{\phi_j\}_{j=1}^N$ be an arbitrary orthonormal basis of $\mathbb{C}^N$ and let $\fg \in \C^N$. For $i=1,\dots,N$, define
\begin{equation}\label{ETGavili}
A_i \fg := A_\alpha^{i-1} \fg \qquad i=1,\dots, N, 
\end{equation}
with $\alpha \in \C^N$ as in Equation \eqref{Ephase}.
For $m,\ell=1,\dots,N$, define $\fg_{m,\ell}$ as in Equation \eqref{Eframe}.
Then we have, for every $f\in{\mathbb C}^N$, 
\[
  A \, \|\mathfrak{f}\|_2^2
    \leq \sum_{m=1}^N \sum_{\ell=1}^N |\langle \mathfrak{f}, \fg_{m,\ell}\rangle|^2
    \leq B \cdot \|\mathfrak{f}\|_2^2
\]
where 
\begin{align*}
A =  N \cdot \min_{k=1}^N  \sum_{j=1}^N |\phi_j(k)|^2\cdot |\widehat{\fg}(j)|^2 \qquad\textrm{and}\qquad B =  N \cdot  \max_{k=1}^N  \sum_{j=1}^N |\phi_j(k)|^2\cdot |\widehat{\fg}(j)|^2.
\end{align*}
Moreover, the constants $A$ and $B$ are sharp.
\end{corollary}
\begin{proof}
Let $\alpha = (\alpha_k)_{k=1}^N \in \C^N$ be given by 
\[
\alpha_k= e^{i\left(c - \frac{2\pi(k-1)}{N}\right)}
\]
for some $c \in [0,2\pi)$. For $1 \leq j \leq N$, let 
\[
f_i := (\alpha_1^{i-1}, \alpha_2^{i-1}, \dots, \alpha_N^{i-1})^\top.
\]
Observe that $A_i = \Phi D_{f_i} \Phi^*$.
Now, for $1 \leq k,\ell \leq N$, we have
\begin{align*}
  \langle f_k, f_\ell \rangle 
    & = \sum_{j=1}^N \alpha_j^{k-1} \overline{\alpha_j^{\ell-1}}
      ={e^{ic(k-l)}} \sum_{j=1}^N e^{-(k-1)i \frac{2\pi(j-1)}{N} + (\ell-1)i \frac{2\pi(j-1)}{N}} \\
    & ={e^{ic(k-l)}} \sum_{j=1}^N e^{-i(k-\ell) \frac{2\pi(j-1)}{N}} 
      ={e^{ic(k-l)}} \sum_{j=1}^N \zeta_j^{k-\ell},
\end{align*}
where $\zeta_j = e^{-\frac{2\pi i(j-1)}{N}}$ is a root of unity. Using the standard orthogonality relations for the discrete Fourier transform, we conclude that 
\[
  \langle f_k, f_\ell \rangle
    = \begin{cases}
      N & \textrm{if }k=\ell \\
      0 & \textrm{otherwise}.
    \end{cases}
\]
The result now follows from Theorem \ref{TframeOrtho} after rescaling the functions $\{f_i\}_{i=1}^N$.
\end{proof}

\section{Discrete frames for  Cayley Graphs}
\label{EigenbasisForQuasiAbelian}

We now examine how Gabor-type frames behave for signals defined on  Cayley graphs.
Given a finite (not necessarily abelian) group $G$ and a subset $S\subset G$, the \emph{Cayley graph} $\cay(G;S)$ is the graph whose vertex set is indexed by the elements of $G$, with adjacency defined as $(x,y)\in E$ if and only if $x^{-1}y\in S$.
If $S$ is symmetric (\textit{i.e.},~$S^{-1} = S$) then the graph is undirected.
A Cayley graph is called \emph{normal} if $S$ is closed under conjugation (\textit{i.e.},~$gSg^{-1} = S$ for all $g\in G$).
Observe that Cayley graphs are regular of degree $|S|$.
As a result, the eigenvectors of both the adjacency and Laplacian matrices of a Cayley graph are the same, so the following analysis applies to either choice of analyzing matrix.
For the remainder of this chapter, we assume that $\Gamma = \cay(G;S)$ is the normal Cayley graph of a finite group $G$ of order $N$.

\subsection{Preferred basis of eigenvectors for normal Cayley graphs}
One major advantage of working with normal Cayley graphs is that their (adjacency or Laplacian) eigenvectors can be written explicitly via the representation theory of the associated group. Recall that a (unitary) representation of $G$ is a homomorphism 
$\pi: G \to \textrm{U}_d(\C)$ from $G$ into the group of $d\times d$ unitary matrices $\textrm{U}_d(\C)$, {\it i.e.}, a map that satisfies
\[
\pi(g_1 g_2) = \pi(g_1) \pi(g_2) \qquad \forall g_1, g_2 \in G. 
\]
The integer $d$ is called the degree of the representation, and will be denoted by $d_\pi$. We denote by $\chi_{\pi}: G \to \C$ the character associated to the representation defined as  $\chi_\pi(g) = \textrm{Tr}(\pi(g))$. 

Let $\widehat G=\{\pi^{(k)}\}_{k=1}^D$ denote the set of (equivalence classes of) irreducible unitary representations of $G$. If $\pi^{(k)}$ has degree $d$ and $1\leq i,j\leq d$, 
we let $\pi^{(k)}_{i,j}:G\rightarrow {\mathbb C}$ denote the coordinate functionals of $\pi^{(k)}$ defined as $\pi^{(k)}_{i,j}(g)=\langle\pi^{(k)}(g)e_j,e_i\rangle$.
Clearly, $\pi^{(k)} = (\pi^{(k)}_{i,j})_{i,j=1}^d$. It is known that these coordinate functionals satisfy the {\it Schur orthogonality relations}: 
\[
\sum_{g \in G} \overline{\pi^{(j)}_{n,m}}(g) \pi^{(k)}_{n',m'}(g) = \delta_{jk} \delta_{nn'} \delta_{mm'} \frac{N}{d_{\pi^{(j)}}},
\]
where $\delta_{ij}$ is the Kronecker delta function. 
As a consequence, the following scaled coordinate functionals
\begin{equation}\label{ECayleyEigenBasis}
\phi_{i,j}^{(k)}:=\sqrt{\frac{d_{{\pi^{(k)}}}}{N}}\left(\pi_{i,j}^{(k)}(g_1), \pi_{i,j}^{(k)}(g_2), \ldots, \pi_{i,j}^{(k)}(g_N) \right)^\top
\end{equation}
form an orthonormal basis for ${\mathbb C}^N$. It turns out that these vectors are precisely the eigenvectors of the adjacency matrix of $\cay(G;S)$ in the case where $S$ is closed under conjugation, {\it i.e.}, $S$ is a union of conjugacy classes of $G$. For a discussion on eigenvalues of  the adjacency matrix of $\cay(G;S)$, without direct calculations with eigenvectors, see \cite{Babai:1979:SpectraOfCayleyGraphs}; for a proof in the case where $S$ is symmetric, see   \cite[Proposition 6.3.1]{MR2882891}. For a statement and proof matching our notations here, see \cite[Theorem III.1]{MDK:2019:Sampta}. Note that the proof of~\cite[Theorem III.1]{MDK:2019:Sampta} does not require the symmetry condition on the generating set.

\begin{theorem}[{cf.~\cite[Theorem III.1]{MDK:2019:Sampta}}]\label{TCayleySpectrum}
Let $\Gamma = \textnormal{Cay}(G;S)$ be the Cayley graph of a finite group $G$ and assume $S$ is closed under conjugation. Then for all $k=1,\dots, D$ and all $1 \leq i,j \leq d_{\pi^{(k)}}$, we have
\[
A\phi_{i,j}^{(k)} = \left(\frac{1}{d_\pi}\sum_{g\in S} \chi_\pi(g)\right)\phi_{i,j}^{(k)},
\]
{where $A$ is the adjacency matrix of $\Gamma$.}
\end{theorem}

\vspace{.25em}
\noindent
Note that we do not assume the graph is undirected in this section, as normal Cayley graphs are not necessarily generated from inverse closed sets.
However, Theorem \ref{TCayleySpectrum} shows that these graphs are always diagonalizable.

\subsection{Frame bounds for normal Cayley graphs}
We now revisit the frame construction given in Theorem \ref{TframeBound}, in the case where $\Gamma$ is a normal  Cayley graph and the eigenbasis  of its  adjacency (or Laplacian) matrix is given as proposed in Theorem \ref{TCayleySpectrum}, 
{\it i.e.} by the following set:  
\begin{equation}\label{eq:ortho-basis}
\left\{\phi_{i,j}^\pi:=\sqrt{\frac{d_{\pi}}{N}}\pi_{i,j}:{\pi\in\widehat{G}, 1\leq i,j\leq d_\pi}\right\}.
\end{equation}
When the representations of $G$ are listed as $\pi^{(1)}, \ldots, \pi^{(D)}$, we use $ \phi_{i,j}^{(k)}$ to denote $\phi_{i,j}^{\pi^{(k)}}$.
As usual, we denote by $\Phi$ the unitary matrix whose columns are the vectors $\phi_{i,j}^\pi$.
Recall the important special case where the translation operators are defined using Fourier multipliers, {\it i.e.} when they are diagonal in the above orthonormal basis.
In this case, we can apply Corollary \ref{CDiagTight} to show that a tight frame is always obtained when the multipliers $\{f_j\}_{j=1}^N$ are orthonormal and $\widehat{\fg}$ is constant, 
\textit{i.e.},~$\fg=c \sum_{\pi\in\widehat{G}, 1\leq i,j\leq d_\pi} \phi_{i,j}^\pi$ for some $c\in \mathbb C$.
We now show how the latter assumption can be considerably relaxed when working on normal Cayley graphs, by exploiting the supplementary structure of the group representations. 
\begin{theorem}\label{TCayleyDiag}
Let $\Gamma = \textnormal{Cay}(G;S)$ be the Cayley graph of a finite group $G$ of order $N$, where the set $S$ is closed under conjugation. 
Equip ${\mathbb C}^N$ with the orthonormal basis $\{\phi_{i,j}^\pi\}_{i,j,\pi}$ as in \eqref{eq:ortho-basis}.
Assume  that $\widehat{\fg}$ is constant over every representation of $G$, {\it i.e.}, for every $\pi \in \widehat{G}$ and $i,j=1,\dots, d_\pi$,
\[
\widehat{\fg}(\phi^\pi_{i,j}) = \widehat{\fg}_\pi
\]
for some constant $ \widehat{\fg}_\pi \in \C$ that depends only on $\pi$ ({\it i.e.} it is independent of $i$ and $j$).  Let $\{f_j\}_{j=1}^N$ be an orthonormal basis of $\C^N$ and define
\[
A_i = \Phi D_{f_i} \Phi^* \qquad (i=1,\dots,N).
\]
Then the family of vectors $\{\fg_{m,\ell} : m,\ell=1,\dots,N\}$ defined as in Equation \eqref{Eframe} forms a tight frame with optimal frame bounds $A = B = \frac{1}{N}\sum_{\pi \in \widehat{G}} |\widehat{\fg}_\pi|^2 d_\pi^2$.  
\end{theorem}
\begin{proof}
We compute the frame bounds given in Corollary \ref{CframeOrtho}. We have 
\begin{align*}
\sum_{\pi\in\widehat{G}}\sum_{i,j=1}^{d_\pi} |\phi_{i,j}^\pi(k)|^2\cdot |\widehat{\fg}(\phi_{i,j}^\pi)|^2 &= \sum_{\pi \in \widehat{G}} \sum_{i,j=1}^{d_\pi} \frac{d_\pi}{N} |\pi_{i,j}(k)|^2 {|\widehat{\fg}_\pi|^2} \\
&=\sum_{\pi \in \widehat{G}} \frac{d_\pi}{N}|\widehat{\fg}_\pi|^2 \sum_{i,j=1}^{d_\pi} |\pi_{i,j}(k)|^2 \\
&= \sum_{\pi \in \widehat{G}} \frac{d_\pi}{N}|\widehat{\fg}_\pi|^2 \sum_{i,j=1}^{d_\pi} \pi_{j,i}(k^{-1}) \pi_{i,j}(k) \\
&=  \sum_{\pi \in \widehat{G}} \frac{d_\pi}{N}|\widehat{\fg}_\pi|^2 \sum_{j=1}^{d_\pi} [\pi(k^{-1}) \pi(k)]_{j,j} \\ 
&= \frac{1}{N}\sum_{\pi \in \widehat{G}} |\widehat{\fg}_\pi|^2 d_\pi^2,
\end{align*}
where the penultimate equation holds since $\pi$ is a unitary representation.
Given that the right hand side of the above equation is independent of $k$, the optimal frame bounds given by Corollary \ref{CframeOrtho} are equal and the frame is tight. 
\end{proof}

\begin{remark}
Recall that all the eigenvectors associated to a given representation $\pi \in \widehat{G}$ correspond to the same eigenvalue $\frac{1}{d_\pi} \sum_{g\in S}{\chi_\pi(g)}$ (see Theorem \ref{TCayleySpectrum}). 
However, different representations may be associated with the same eigenvalue. 
Thus, our assumption in Theorem \ref{TCayleyDiag} is a weaker condition than the one in \cite{Shuman:2016:VertexFrequencyAnalysisOnGraphs}, where $\widehat{\fg}$ needs to be constant on each eigenspace. 
\end{remark}

Applying Theorem \ref{TCayleyDiag} to the translation operators given by Equations \eqref{ETShuman} and \eqref{ETGavili}, we immediately obtain the following families of tight frames for Cayley graphs. 

\begin{corollary}\label{CCaleyTight}
Let $\Gamma = \textnormal{Cay}(G;S)$ be the Cayley graph of a finite group $G$ of order $N$, where $S$ is closed under conjugation. 
Equip ${\mathbb C}^N$ with the orthonormal basis $\{\phi_{i,j}^\pi\}_{i,j,\pi}$ as in \eqref{eq:ortho-basis}.
Assume $\widehat{\fg}$ is constant over every representation of $G$, {\it i.e.}, for every $\pi \in \widehat{G}$ and $i,j=1,\dots, d_\pi$, 
\[
\widehat{\fg}({\phi^\pi_{ij}}) = \widehat{\fg}_\pi
\]
for some constant $ \widehat{\fg}_\pi \in \C$ that is independent of $i$ and $j$. Let $\{A_i\}_{i=1}^N$ be either 
\begin{enumerate}
\item the translation operators $T_i \fg = \fg * (\sqrt{N}\delta_i)$ as in Corollary \ref{CTransShuman}, or 
\item the repeated shifts $A_i \fg = A_\alpha^{i-1} \fg$ as in Corollary \ref{CTransGavili}.
\end{enumerate}
Then the family of vectors $\{\fg_{m,\ell} : m,\ell=1,\dots,N\}$ defined as in Equation~\eqref{Eframe} forms a tight frame with optimal frame bounds $A = B = \sum_{\pi \in \widehat{G}} |\widehat{\fg}_\pi|^2 d_\pi^2$. 
\end{corollary}

\subsection{More general translations on Cayley graphs}
We conclude this paper with a short discussion on the concept of graph translations. We present natural candidates for translations; such translations may then be used to produce frames for graph signals.
Notice that every Cayley graph comes equipped with its own natural notion of translation, via multiplication by a group element. As a consequence, given $g \in G$, we can translate a signal $\ff: G \to \C$ by:
\[
\ff(h) \mapsto \ff(g\cdot h) \qquad (h \in G).
\]
Equivalently, the above translation is given by the action of the left (or the right) regular representation: 
\[
L(g)\ff(h) = \ff(g^{-1} h) \qquad (h \in G).
\]
Our next result shows that the translation operators $T_i$ given by Equation \eqref{Etranslation} are essentially equivalent to the action of the left regular representation. 
\begin{theorem}\label{thm:Translations can be computed by characters}
Let $\Gamma = \textnormal{Cay}(G;S)$ be the Cayley graph of a finite group $G$ of order $N$ with $S$ closed under conjugation, and equipped with the eigenbasis $\phi^{\pi}_{i,j}$ given by Equation \eqref{ECayleyEigenBasis}. Assume $\widehat{\fg}$ is constant over every representation of $G$, {\it i.e.}, for every $\pi \in \widehat{G}$ and $i,j=1,\ldots, d_\pi$, 
\[
\widehat{\fg}(\phi^\pi_{ij}) = \widehat{\fg}_\pi \in \C.
\]
Then the graph translation operator $T_\ell$ given in Equation \eqref{Etranslation} is given by
\[ 
  (T_\ell \mathfrak g)(k) 
    = \frac{1}{\sqrt{N}} \sum_{\pi \in \widehat G} d_\pi \widehat{\mathfrak g}(\pi)
      \chi_\pi(\ell^{-1}k),
\]
where $L$ is the left regular representation of $G$ and $e$ is the group identity.
\end{theorem}

\begin{proof}
Using the definition of translations in Equation \eqref{Etranslation} and the fact that every representation $\pi$ in the sum is unitary, we have
{\allowdisplaybreaks\begin{align*}
  (T_\ell \fg)(k)
    & = \sqrt{N} \sum_{\pi\in \widehat G}
      \sum_{j=1}^{d_\pi} 
      \sum_{i=1}^{d_\pi}
      \frac{d_\pi}{N}
      \widehat \fg (\phi^\pi_{i,j})
      \overline{\pi_{i,j}}(\ell) \pi_{i,j}(k) \\
    & = \frac{1}{\sqrt{N}} \sum_{\pi\in \widehat G} d_\pi \widehat{\fg}(\pi)
      \sum_{j=1}^{d_\pi}
      \sum_{i=1}^{d_\pi}
      \pi_{j,i}(\ell^{-1}) \pi_{i,j}(k) \\
    & = \frac{1}{\sqrt{N}} \sum_{\pi\in \widehat G} d_\pi \widehat{\fg}(\pi)
      \sum_{j=1}^{d_\pi} [\pi(\ell^{-1})\pi(k)]_{j,j}. \\
      \end{align*}}
      Since each $\pi$ is a homomorphism, we get 
$(T_\ell \fg)(k) = \frac{1}{\sqrt{N}} \sum_{\pi\in \widehat G} d_\pi \widehat{\fg}(\pi)  \chi_\pi(\ell^{-1}k). $
\end{proof}

\begin{corollary}
\label{cor:TranslationInvarianceUnderShifting}
Under the conditions of Theorem \ref{thm:Translations can be computed by characters}, the
translation operators $T_\ell$ for normal Cayley graphs
are invariant when shifted in both indices, that is, for all $m\in G$
\[ (T_{\ell} \fg)({k}) = (T_{\ell m} \fg)({k m}) = (T_{m \ell} \fg)({mk}). \]
In particular, choosing $m = \ell^{-1}$, we see that
\[ (T_\ell \fg)(k) = (T_e \fg)({\ell^{-1}k}) = L(\ell)[T_e \fg](k), \]
where $e$ is the group identity element and $L$ is the left regular representation of $G$.
\end{corollary}

\begin{proof}
The proof is an immediate consequence of Theorem \ref{thm:Translations can be computed by characters} and the fact that characters are class functions.
\end{proof}

\begin{remark}
Notice that Theorem \ref{thm:Translations can be computed by characters} immediately implies that, when translation is given by the operators $T_\ell$, the sharp frame bounds for the associated frame are $A = B = \|T_e \fg\|_2^2$.
A simple calculation shows that 
\[
  \|T_e \fg\|_2^2 = \sum_{\pi \in \widehat{G}} |\widehat{\fg}_\pi|^2 d_\pi^2, 
\]
recovering the expression for the frame bounds given in Corollary \ref{CCaleyTight}.
\end{remark}

\begin{remark}
Theorem \ref{thm:Translations can be computed by characters} and Corollary \ref{cor:TranslationInvarianceUnderShifting} show that for
  $\mathfrak{g}$ defined spectrally as in \cite{Shuman:2012:WGFT,Shuman:2016:VertexFrequencyAnalysisOnGraphs} so that they are constant on eigenspaces, the behavior of the translation operator reduces to
  $T_\ell\mathfrak g(k) = L(\ell)T_e\mathfrak g(k)$.
  Then $T_e$ can be viewed as some pre-processing of the original window function $\fg$, which is then translated by the usual group translation.
  In particular, this shows that with this particular choice of basis, by Corollary \ref{CCaleyTight}, one can always obtain tight frames for Cayley graphs using translations which respect the graph structure.
\end{remark}

A `natural' choice of translations for graphs would be operators which permute the vertex set while perfectly respecting the structure of the graph.
That is, each translation $T$ should satisfy  $(Tx,Ty)$ is an edge in $\Gamma$ if and only if $(x,y)$ is.
In other words, translation preserves the adjacency structure of $\Gamma$.
This is precisely the definition of a \emph{graph automorphism}, and for Cayley graphs, the collection $\{L(g)\}_{g\in G}$ is in fact contained in the collection of all graph automorphisms. 
For the sake of completeness, we include this well-known fact and its proof below. 

\begin{theorem}[Natural choice of translations]
\label{thm:LsubgroupOfAutomorphisms}
  Let $\Gamma=\textnormal{Cay}(G;S)$ for a finite group $G$ and any generating set $S$.
  Then the automorphism group of $\Gamma$ contains the family $\{L(g)\}_{g\in G}$ as a subgroup.
\end{theorem}

\begin{proof}

  It is sufficient to show that $(x,y)$ is an edge in $ \Gamma$ if and only if $(g^{-1}x,g^{-1}y)$ is an edge in $ \Gamma$ for every element $g\in G$.
  However, this is clear as $(x,y)$ is an edge if and only if $S\ni x^{-1}y=x^{-1}gg^{-1}y = (g^{-1}x)^{-1}(g^{-1}y)$, therefore $(g^{-1}x,g^{-1}y)$ is also an edge.
  Thus $L(g)$ is an automorphism of $ \Gamma$.
  As the left regular representation $L$ is a group homomorphism, it is obvious that its image is a subgroup.
\end{proof}
\begin{remark}
In Theorem \ref{thm:LsubgroupOfAutomorphisms}, we do not restrict ourselves to normal Cayley graphs.
In the special case of Cayley graphs with the generating set closed under conjugation, the collection $\{R(g)\}_{g\in G}$ is also a subgroup of the automorphism group of the graph.
The proof is similar, showing $(x,y)$ is an edge if and only if $(xg,yg)$ is an edge, which follows as above and from the additional fact that $S$ is closed under conjugation.
\end{remark}

Recall that the eigenvectors of a given graph are not uniquely determined when the graph has repeated eigenvalues, as one must pick a basis for each eigenspace. We now demonstrate how this choice can dramatically impact the properties of the resulting translation operators and frames.

\begin{figure}[ht!]
  \begin{tikzpicture}
    \begin{scope} [vertex style/.style={draw,
                                       circle,
                                       minimum size=6mm,
                                       inner sep=0pt,
                                       outer sep=0pt
                                       }] 
      \path \foreach \i in {0,...,2}{%
       (135+45*\i:2) coordinate[vertex style] (a\i) node {\pgfmathparse{mod(2*\i,6)}${\pgfmathprintnumber{\pgfmathresult}}$}
       (225-45*\i:2) node[below left, xshift=-3pt, yshift=-3pt] {\footnotesize \pgfmathparse{mod(\i+4,6)+1}$v_{\pgfmathprintnumber{\pgfmathresult}}$}
       (45-45*\i:2) coordinate[vertex style] (b\i) node {\pgfmathparse{mod(2*\i+0,6)+1}${\pgfmathprintnumber{\pgfmathresult}}$}%
       (45-45*\i:2) node[below right, xshift=3pt, yshift=-3pt] {\footnotesize \pgfmathparse{mod(\i+1,6)+1}$v_{\pgfmathprintnumber{\pgfmathresult}}$}%
       }
       ; 
    \end{scope}
     \begin{scope} [edge style/.style={draw=gray, thick}]
       \foreach \i  in {0,...,2}{%
       \foreach \j  in {0,...,2}{%
       \draw[edge style] (a\i)--(b\j);
       }}
     \end{scope}
\begin{scope}[xshift=7cm]
\matrix (M)[%
        row sep=0mm,
        matrix of math nodes,
        left delimiter={(},
        right delimiter={)},
        column sep=1mm,
        inner sep=-1pt,
        nodes={inner sep=.2em}, 
        ]{ 
             \displaystyle\frac{1}{\sqrt{6}}
               & \displaystyle\frac{1}{\sqrt{6}} 
               & \displaystyle\frac{1}{\sqrt{6}} 
               & \displaystyle\frac{1}{\sqrt{6}} 
               & \displaystyle\frac{1}{\sqrt{6}} 
               & \displaystyle\frac{1}{\sqrt{6}} \\
             \displaystyle\frac{1}{\sqrt{6}}
               & \displaystyle\frac{\gamma}{\sqrt{6}}
               & \displaystyle\frac{\gamma^2}{\sqrt{6}}
               & \displaystyle\frac{\gamma^3}{\sqrt{6}}
               & \displaystyle\frac{\gamma^4}{\sqrt{6}}
               & \displaystyle\frac{\gamma^5}{\sqrt{6}} \\
             \displaystyle\frac{1}{\sqrt{6}}
               & \displaystyle\frac{\gamma^3}{\sqrt{6}}
               & \displaystyle\frac{1}{\sqrt{6}} 
               & \displaystyle\frac{\gamma^3}{\sqrt{6}}
               & \displaystyle\frac{1}{\sqrt{6}} 
               & \displaystyle\frac{\gamma^3}{\sqrt{6}} \\
             \displaystyle\frac{1}{\sqrt{6}}
               & \displaystyle\frac{\gamma^5}{\sqrt{6}}
               & \displaystyle\frac{\gamma^4}{\sqrt{6}}
               & \displaystyle\frac{\gamma^3}{\sqrt{6}}
               & \displaystyle\frac{\gamma^2}{\sqrt{6}}
               & \displaystyle\frac{\gamma}{\sqrt{6}} \\
             \displaystyle\frac{1}{\sqrt{6}}
               & \displaystyle\frac{\gamma^4}{\sqrt{6}}
               & \displaystyle\frac{\gamma^2}{\sqrt{6}}
               & \displaystyle\frac{1}{\sqrt{6}} 
               & \displaystyle\frac{\gamma^4}{\sqrt{6}}
               & \displaystyle\frac{\gamma^2}{\sqrt{6}} \\
             \displaystyle\frac{1}{\sqrt{6}} 
               & \displaystyle\frac{\gamma^2}{\sqrt{6}}
               & \displaystyle\frac{\gamma^4}{\sqrt{6}}
               & \displaystyle\frac{1}{\sqrt{6}}
               & \displaystyle\frac{\gamma^2}{\sqrt{6}} 
               & \displaystyle\frac{\gamma^4}{\sqrt{6}} \\
};
\node[anchor=south east] (cornernode) at (M-1-1.north west) {}; 
\foreach[count=\xi] \x in {0,1,3,5,4,2}{ 
\node[anchor=center, xshift=-18pt, scale=.7] (M-\xi-0) at (cornernode |- M-\xi-1) {\x}; 
}
\foreach[count=\xi] \x in {$\chi_0$, $\chi_1$, $\chi_2$, $\chi_3$, $\chi_4$, $\chi_5$}{ 
\node[yshift=4pt, scale=.7] (M-0-\xi) at (cornernode -| M-1-\xi) {\x}; 
}
\end{scope}
  \end{tikzpicture}
  \caption{The Graph $K_{3,3}$ and basis of coefficient functions of $\mathbb{Z}_6$. $\gamma:=\exp[\frac{2\pi i}{6}]$.}
  \label{fig:K33Z6}
\end{figure}
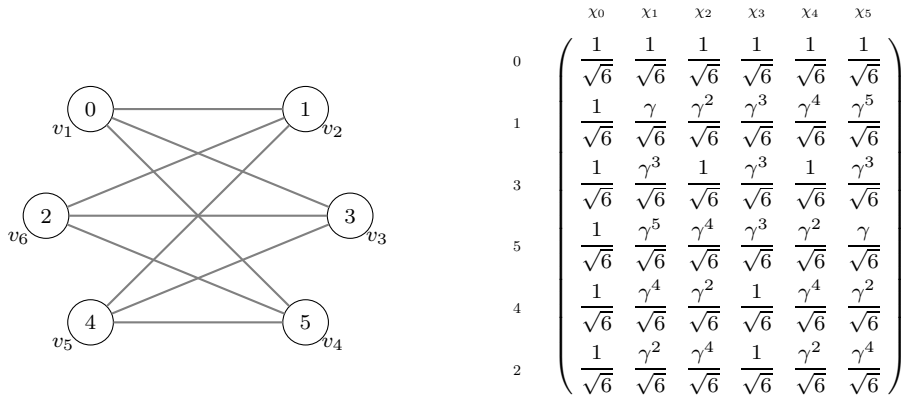

\begin{example}
\label{example:k33}

Consider the graph $K_{3,3}$, the complete bipartite graph on 6 vertices with equal partitions (Figures \ref{fig:K33Z6} and \ref{fig:K33S3}).
This graph can be realized as the Cayley graph of $\mathbb{Z}_6$ or $S_3$.
Depending on the group realization, the eigenvectors chosen from the group representations are considerably different.
While it might seem desirable to choose the group to be $\mathbb{Z}_6$, unless we are considering a time series discretization of signals on the real line, the underlying reality of our graph is unlikely to be well-modeled by an abelian group, so it seems unlikely to capture the desired behavior of our irregular domain.

One consequence of these different eigenbases is that for non-isometric translation operators such as the one given in Equation \eqref{Etranslation},
the frame bounds can vary based on the choice of the eigenbasis for a fixed graph and fixed window function.
Recall that translation in this case can be written as 
\[
T_j\mathfrak f:= \sqrt{N}\Phi(\Phi^*\mathfrak{f} \circ \Phi^*\delta_j), 
\]
where this form makes explicit the dependence on the chosen eigenbasis $\Phi$.

Then the matrix of characters described for $\mathbb{Z}_6$ (Figure \ref{fig:K33Z6}) diagonalizes the graph adjacency matrix (or Laplacian).
As all the entries lie in the $\frac{1}{\sqrt{N}}$-radius circle, translation is an isometry leading to a tight frame for any window function.
However, another basis can be chosen as this is the graph $\textnormal{Cay}(S_3; \{(12),(13),(23)\})$ (Figure \ref{fig:K33S3}).
Then the coefficient functions of the unitary, irreducible representations of $S_3$ also diagonalize the associated matrices, and in this case the translation operator is not an isometry.

In the basis for $\mathbb{Z}_6$, any non-zero function $\mathfrak g$ will provide a tight frame, meaning the ratio between frame bounds $\frac{B}{A}=1$.
However, taking the window as $\mathfrak g:= \frac{1}{7}(6,3,2,0,0,0)^\top$, in the basis provided in Figure \ref{fig:K33S3}, the ratio is $\frac{69}{29}\approx 2.4$.
We remark that any $\mathfrak g$ such that $\widehat{\mathfrak g}$ is not constant on representations should yield similar examples where the frame is tight in the first basis and not in the second. This provides a stark reminder that one should be careful when choosing an eigenbasis for a given graph. 
The problem of choosing a suitable eigenbasis for a given graph appears to be a challenging task, which we hope to investigate in future work.

\begin{figure}
  \begin{tikzpicture}
    \begin{scope} [vertex style/.style={draw,
                                       circle,
                                       minimum size=6mm,
                                       inner sep=0pt,
                                       outer sep=0pt
                                       }] 
      \path \foreach \i in {0,...,2}{%
       (225-45*\i:2) node[below left, xshift=-3pt, yshift=-3pt] {\footnotesize \pgfmathparse{mod(\i+4,6)+1}$v_{\pgfmathprintnumber{\pgfmathresult}}$}
       (45-45*\i:2) node[below right, xshift=3pt, yshift=-3pt] {\footnotesize \pgfmathparse{mod(\i+1,6)+1}$v_{\pgfmathprintnumber{\pgfmathresult}}$}%
       }
       ; 

      \node[vertex style] at (a0) {e};
      \node[vertex style] at (a1) {\footnotesize 123};
      \node[vertex style] at (a2) {\footnotesize 132};
      \node[vertex style] at (b0) {\footnotesize 12};
      \node[vertex style] at (b1) {\footnotesize 13};
      \node[vertex style] at (b2) {\footnotesize 23};
    \end{scope}
     \begin{scope} [edge style/.style={draw=gray, thick}]
       \foreach \i  in {0,...,2}{%
       \foreach \j  in {0,...,2}{%
       \draw[edge style] (a\i)--(b\j);
       }}
     \end{scope}
\begin{scope}[xshift=6.25cm]
\matrix (M)[%
        row sep=-1.5mm,
        matrix of math nodes,
        left delimiter={(},
        right delimiter={)},
        column sep=1mm,
        inner sep=-1pt,
        nodes={inner sep=.2em}, 
        ]{ 
             \displaystyle\frac{1}{\sqrt{6}} &
             \displaystyle\frac{1}{\sqrt{6}} &
             \displaystyle\frac{\sqrt{2}}{\sqrt{6}} &
             0 &
             0 &
             \displaystyle\frac{\sqrt{2}}{\sqrt{6}}
             \\[.5em]
             \displaystyle\frac{1}{\sqrt{6}} &
             -\displaystyle\frac{1}{\sqrt{6}} &
             0 & 
             \displaystyle\frac{\sqrt{2}}{\sqrt{6}} &
             \displaystyle\frac{\sqrt{2}}{\sqrt{6}} &
             0
             \\[.5em]
             \displaystyle\frac{1}{\sqrt{6}} &
             -\displaystyle\frac{1}{\sqrt{6}} &
             0 &
             \displaystyle\frac{\sqrt{2}\omega^2}{\sqrt{6}} &
             \displaystyle\frac{\sqrt{2}\omega}{\sqrt{6}} &
             0
             \\[.5em]
             \displaystyle\frac{1}{\sqrt{6}} &
             -\displaystyle\frac{1}{\sqrt{6}} & 
             0 & 
             \displaystyle\frac{\sqrt{2}\omega}{\sqrt{6}} &
             \displaystyle\frac{\sqrt{2}\omega^2}{\sqrt{6}} &
             0
             \\[.5em]
             \displaystyle\frac{1}{\sqrt{6}} &
             \displaystyle\frac{1}{\sqrt{6}} &
             \displaystyle\frac{\sqrt{2}\omega}{\sqrt{6}} &
             0 & 
             0 & 
             \displaystyle\frac{\sqrt{2}\omega^2}{\sqrt{6}}
             \\[.5em]
             \displaystyle\frac{1}{\sqrt{6}} &
             \displaystyle\frac{1}{\sqrt{6}} &
             \displaystyle\frac{\sqrt{2}\omega^2}{\sqrt{6}} &
             0 & 
             0 & 
             \displaystyle\frac{\sqrt{2}\omega}{\sqrt{6}}\\
};
\node[anchor=south east] (cornernode) at (M-1-1.north west) {}; 
\foreach[count=\xi] \x in {e,(12),(13),(23),(123),(132)}{ 
\node[anchor=center, xshift=-12pt, scale=.7] (M-\xi-0) at (cornernode |- M-\xi-1) {\x}; 
}
\foreach[count=\xi] \x in {$\chi_0$, $\chi_\pm$, $\pi_{1,1}$, $\pi_{2,1}$, $\pi_{1,2}$, $\pi_{2,2}$}{ 
\node[yshift=4pt, scale=.7] (M-0-\xi) at (cornernode -| M-1-\xi) {\x}; 
}
\end{scope}
  \end{tikzpicture}
  \caption{The Graph $K_{3,3}$ and basis of coordinate functionals of $S_3$. $\omega:=\exp[\frac{2\pi i}{3}]$.}
  \label{fig:K33S3}
\end{figure}
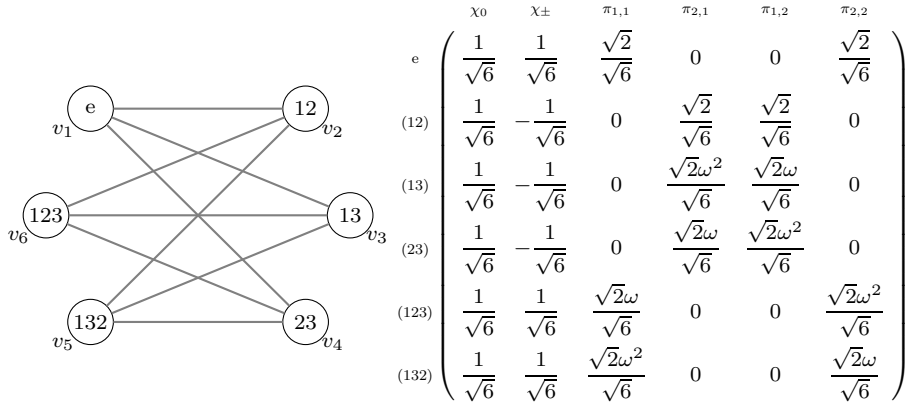

\end{example}

\begin{acknowledgements}
M. Ghandehari was supported by NSF grant (DMS--1902301) while this project was being completed. 
D. Guillot was partially supported by a collaboration grant for mathematicians from the Simons Foundation (\#526851), and by  a Strategic Initiative grant from the University of Delaware Research Foundation (\#18A00532).
We sincerely thank the anonymous reviewers for reading the manuscript and suggesting improvements.
\end{acknowledgements}

%
%


\end{document}